\documentclass[12pt, reqno]{article}
\usepackage{amscd,amsmath,amsthm,amssymb,graphics}
\usepackage{amsfonts,amssymb,amscd,amsmath,enumitem,verbatim}
\usepackage[a4paper,top=2cm,left=1.5cm,right=1.5cm]{geometry}
\usepackage{xcolor}
\usepackage{csquotes}
\theoremstyle{plain}
\usepackage{booktabs}
\usepackage{hyperref}
\usepackage{authblk}

\newtheorem{theorem}{Theorem}[section]
\newtheorem{lemma}[theorem]{Lemma}
\newtheorem{corollary}[theorem]{Corollary}
\newtheorem{proposition}[theorem]{Proposition}

\newtheorem{question}[theorem]{Question}
\theoremstyle{definition}
\newtheorem{definition}[theorem]{Definition}
\newtheorem{remark}[theorem]{Remark}

\newtheorem{example}[theorem]{Example}

\newcommand{\N}{\mathbb{Z}^{+}}
\newcommand{\Z}{\mathbb{Z}}
\newcommand{\Q}{\mathbb{Q}}
\newcommand{\R}{\mathbb{R}}

\renewcommand{\O}{\mathcal{O}}

\newcommand{\CC}{\mathcal{C}}
\newcommand{\Ccl}{\CC^{\mathrm{cl}}}
\newcommand{\Cdiag}{\CC^{\mathrm{diag}}}
\newcommand{\Cfr}{\CC^{\mathrm{fr}}}
\newcommand{\Cclfr}{\CC^{\mathrm{cl,fr}}}

\newcommand{\nc}{\mathrm{nc}}
\newcommand{\cl}{\mathrm{cl}}
\newcommand{\diag}{\mathrm{diag}}

\newcommand{\mdiag}{m^{\mathrm{diag}}}

\newcommand{\OPlus}{\O^{+}}
\newcommand{\U}{\mathcal{U}}
\newcommand{\UPlus}{\U^{+}}
\newcommand{\p}{\mathfrak{p}}
\newcommand{\qf}[1]{\langle #1 \rangle}
\newcommand{\OO}{\OPlus_K/\U_K^2}
\DeclareMathOperator{\Tr}{Tr}
\DeclareMathOperator{\NN}{N}
\let\epsilon\varepsilon
\let\phi\varphi

\newcommand{\keywords}[1]{\noindent\textbf{Keywords:} #1}
\newcommand{\subjclass}[2]{\noindent\textbf{#1 Mathematics Subject Classification:} #2.}

\title{Universality criterion sets for quadratic forms\\ over number fields}

\author[1,a]{V\' \i t\v ezslav Kala}
\author[2,b]{Jakub Kr\'asensk\'y}
\author[1,3,c]{Giuliano Romeo}

\affil[1]{Charles University, Faculty of Mathematics and Physics, Department of Algebra, Sokolovsk\'{a} 83, 186 75 Praha 8, Czech Republic}

\affil[2]{Czech Technical University in Prague, Faculty of Information Technology, Department of Applied Mathematics, Thákurova~9, 160~00 Praha~6, Czech Republic}

\affil[3]{Politecnico of Torino, Department of Mathematical Sciences Giuseppe Luigi Lagrange, Corso Duca degli Abruzzi 24, Torino, Italy}

\affil[a]{vitezslav.kala@matfyz.cuni.cz}

\affil[b]{jakub.krasensky@fit.cvut.cz}

\affil[c]{giuliano.romeo@polito.it}

\begin{document}

\maketitle

\begin{abstract} In analogy with the 290-Theorem of Bhargava--Hanke, 
a criterion set is a finite subset $C$ of the totally positive integers in a given totally real number field such that if a quadratic form represents all elements of $C$, then it necessarily represents all totally positive integers, i.e., is universal. We use a novel characterization of minimal criterion sets to show that they always exist and are unique, and that they must contain certain explicit elements. We also extend the uniqueness result to the more general setting of representations of a given subset of the integers.

\keywords{universal quadratic form, quadratic lattice, totally real number field, criterion set}
\subjclass{2020}{11E12, 11E20, 11E25, 11R04, 11R80}
\end{abstract}

\section{Introduction}

The representations of integers by quadratic forms present numerous problems that have significantly influenced the development of mathematics from its pre-history till today. In particular, \textit{universal} quadratic forms over a totally real number field $K$ (i.e., those that represent all totally positive algebraic integers) provide several mysterious arithmetic invariants of $K$, such as the \textit{minimal rank} of a universal form and a universality \textit{criterion set}. The latter is a finite set of integers whose representability by a given form implies the universality of the form, such as the sets $\{1,2,3,\dots,15\}$ and $\{1,2,3,\dots,290\}$ from the famous 15- and 290-Theorems for $K=\Q$ of Conway--Schneeberger and Bhargava--Hanke \cite{Bh, BH}.

In order to be more precise, let us continue the discussion in the geometric language of \textit{quadratic lattices} and denote by $m_K$ the minimal rank of a universal quadratic lattice over $K$
(see Section \ref{se:prel} for any undefined notions). This yields a slight generalization as quadratic forms correspond to \textit{free} lattices.

While the fact that $m_\Q=4$ is cleanly explained by the universality of the sum of 4 squares $x^2+y^2+z^2+w^2$ and a local argument for the non-universality of ternary forms, we do not know a clear reason for Bhargava--Hanke's constant 290, except for the actual, rather involved and computational, proof (perhaps because there is simply none).

Similarly, over number fields, we now have a basic understanding of the minimal rank $m_K$ thanks to the efforts of many mathematicians over the last eighty years.
Maa{\ss} \cite{M} showed that 3, the smallest possible value of $m_K$, is attained for $K=\Q(\sqrt 5)$, where the sum of 3 squares is universal. Despite Siegel's theorem \cite{Si} that there are no other cases with universal sum of any number of squares, universal forms always exist thanks to the asymptotic local-global principle of Hsia--Kitaoka--Kneser \cite{HKK}. While Kitaoka's influential conjecture that there are only finitely many fields $K$ with $m_K=3$ remains open, Chan--Kim--Raghavan \cite{CKR} resolved it for \textit{classical} quadratic forms (i.e., under the assumption that the associated bilinear form attains integral values) over real quadratic fields $\Q(\sqrt D)$ -- they admit universal classical ternaries only for $D=2,3,5$. Recently, Kim--Kim--Park \cite{KKP} covered also the non-classical case, and their results have been made more explicit by  Kala--Kr\' asensk\' y--Park--Yatsyna--\. Zmija \cite{KK+}. In general, we only have the 
theorem of  Earnest--Khosravani \cite{EK} that, in odd degrees $[K:\Q]$, there are no universal ternary lattices for local reasons, and the 
weaker result of Kala--Yatsyna \cite{KY} asserting finiteness when the degree $[K:\Q]$ is bounded.
Regardless of the difficulty of studying ternary lattices, the minimal ranks $m_K$ are often large, both over quadratic fields \cite{BK1, Ka1, KYZ} and in higher degrees \cite{KT, KTZ, Man, Ya}. For further related results, see, e.g., \cite{Ea,KL,KY1}, works on indefinite forms \cite{HHX,XZ}, 
or the surveys \cite{Ka2,Ki}.

Universality criterion sets for positive definite lattices are much more inscrutable. It was only a few years ago when Chan--Oh \cite{CO} established their existence, however, without implying anything about their structures or sizes. The only case known explicitly is due to Lee \cite{Le} 
who proved that a classical quadratic lattice over $\Q(\sqrt5)$ is universal if it represents the 7 elements
\[
1,2,\frac{5+\sqrt5}{2}, \frac{7\pm\sqrt5}{2},2\cdot\frac{5+\sqrt5}{2}, 3\cdot\frac{5+\sqrt5}{2}.
\]

However, more is known in the parallel situation of representations of quadratic forms of given rank (e.g., a positive definite quadratic form over $\Z$ is $k$-universal if it represents all positive definite forms over $\Z$ of rank $k$) besides the basic finiteness result for criterion sets, due to Kim--Kim--Oh \cite{KKO2} over $\Z$ and Chan--Oh \cite{CO} again in general. 
Kim--Kim--Oh \cite{KKO} found a 6-element criterion set for 2-universality of classical forms over $\Z$, whose uniqueness was then established by Kominers \cite{Kom}. The criterion set for 8-universality is again known and unique \cite{Oh,Kom2}, but Kim--Lee--Oh showed that they are \textit{not} unique for any $k\geq 9$ \cite{KLO}. This phenomenon of  non-uniqueness was first observed in the context of $S$-universality, i.e., the study of forms representing all elements of a given set $S$ of forms of rank $k$, by Elkies--Kane--Kominers \cite{EKK} who found examples with $k=3$ and $9$.

In the local case, there are actually many results describing the structure of universal forms and criterion sets \cite{Be1,Be2,HH,HHX,XZ}.
Criterion sets are also fairly well understood for Hermitian forms \cite{KKP0} and for sums of polygonal numbers, see, e.g., \cite{JuK, KaL}. Conversely, finite criterion sets do not exist for indefinite forms \cite{XZ} and for higher degree forms \cite{KP}. 

\medskip

Our goal is to significantly advance the state of the art in the study of criterion sets for universality over general totally real number fields $K$.
First of all, in sharp contrast with the preceding discussion, minimal criterion sets \textit{are unique} in this case.

\begin{theorem} \label{th:mainunique}
For every totally real number field $K$, there exists a unique criterion set $\CC_K$ which is minimal with respect to inclusion. This set $\CC_K$ is finite and consists precisely of the critical elements (see Definition $\ref{de:critical}$).
\end{theorem}

This will be proved as a special case of Theorem \ref{th:unique}. 
However, before proceeding further, we must clarify a minor technicality in the statement of the theorem. Let $\O_K^+$ be the set of all totally positive integers in $K$ and $\U_K$ the unit group of $K$.  
Trivially,  a quadratic lattice  represents an element $\alpha\in\O_K^+$  if and only if it represents  $\alpha\epsilon^2$ for any unit $\epsilon\in\U_K$. As representability thus depends only on the class of $\alpha$ modulo squares of units (i.e., on its class in $\OO$), so does the criterion set. Therefore, no finite subset of $\O_K^+$ has a chance of being a \textit{unique} criterion set, e.g., in Lee's example above one could replace the element $1$ by the squared unit $(3+\sqrt 5)/2=((1+\sqrt 5)/2)^2$. To deal with this issue, we consider criterion sets as subsets of $\OO$ throughout the paper (see Definition \ref{de:crit}).

With essentially the same proof, we establish versions of Theorem \ref{th:mainunique} also for universality criterion sets 
\begin{itemize}
    \item $\Ccl_K$ for classical lattices,
    \item $\Cfr_K$ for free lattices = quadratic forms,
    \item $\Cclfr_K$ for classical free lattices = classical quadratic forms,
    \item $\Cdiag_K$ for diagonal lattices = diagonal quadratic forms.
\end{itemize}

\begin{theorem}\label{th:main2}
For every totally real number field $K$, we have 
\begin{equation}\label{eq:crit}
\Cdiag_K\subset \Cclfr_K= \Ccl_K\subset \Cfr_K=\CC_K.    
\end{equation}

All of these sets are closed under conjugation and under multiplication by totally positive units, i.e., if $\CC$ is any of the sets above, $\sigma\in\mathrm{Aut}(K/\Q)$, and $\epsilon\in\U_K^+$, then
$\sigma(\CC)=\epsilon\cdot\CC=\CC$.
\end{theorem}

This will be established as part of Theorems \ref{th:unique} and \ref{th:free} and Corollary \ref{co:closed}.

In Corollary \ref{pr:quadratic} we make the observation that there is a universal diagonal lattice of rank equal $\#\Cdiag_K$. Thus the cardinalities of the criterion sets from \eqref{eq:crit} satisfy
$$m_K\leq 
\#\Cdiag_K\leq \#\Ccl_K\leq \#\CC_K,$$
and so we know that they are often large by the results on large values of $m_K$ discussed above.
In particular, there are only finitely many real quadratic fields $K$ where $\#\Cdiag_K\leq 7$ thanks to the result of Kim--Kim--Park \cite{KKP}.

Since real quadratic fields with totally positive fundamental unit have density 1, Theorem \ref{th:main2} has a curious consequence (Corollary \ref{co:density}) that
for density 1 of real quadratic fields $K$, the criterion sets from \eqref{eq:crit} have even size
(and are arbitrarily large by \cite{KYZ}).

\medskip

We then turn to establishing that certain elements are always contained in the criterion sets. This is in particular the case of \textit{indecomposables}, i.e., elements $\alpha\in\O_K^+$ that cannot be decomposed as $\alpha=\beta+\gamma$ with $\beta,\gamma\in\O_K^+$. Indecomposable elements have long played crucial role in the study of universal forms and lattices, ranging from Siegel \cite{Si}, Kim \cite{Ki2}, to almost all of the recent works discussed above in relation to the minimal rank $m_K$. There has also been considerable interest in estimating norms and other properties of indecomposables \cite{CL+,DS,JK,Ka1.5,Ti,TV}.

In Theorem \ref{th:indecomposable}, we show that the criterion sets must contain all \textit{squarefree} indecomposables (whose number can be estimated analytically in real quadratic fields \cite{BK2} as a variation of Kronecker’s limit formula). Under certain assumptions on squarefreeness and non-existence of small elements, all the elements $1, 2, 3, 5, 6, 7, 10, 14, 15$ from the Conway--Schneeberger 15-Theorem are also necessarily contained in $\Cdiag_K$ (Proposition \ref{pr:15iscritical}), and likewise for the Bhargava--Hanke 290-Theorem (Proposition \ref{pr:290iscritical}).
Note that the follow-up paper \cite{KR} will develop the theory of escalations over number fields and use it to obtain explicit computational results on criterion sets for $\Q(\sqrt 2),\Q(\sqrt 3),$ and $\Q(\sqrt 5)$.

Although we focused the preceding discussion on the slightly simpler case of universality criterion sets, 
there has been a lot of interest also in the study of quadratic forms representing only a given subset of the integers,
especially over $\Z$ in connection with the conjectural 451-Theorem of Rouse \cite{Ro} and its variants \cite{BC,De,DR}.
Thus, 
throughout the paper we work with $S$-universal lattices, i.e., lattices that represent all elements of a given set $S\subset\OO$; we denote the corresponding criterion sets as $\Cdiag_S \subset \Ccl_S \subset \CC_S\subset S$ (see Definitions \ref{de:crit} and \ref{def:C_S}).
In particular, Theorems \ref{th:unique} and \ref{th:free} extend our aforementioned main theorems to this setting.

In Theorem \ref{th:generalized} we then show that one can in fact remove the condition $\CC \subset S$ from the definition of the criterion set and still obtain the same uniqueness result; this was not previously known even for $K=\Q$.

We conclude the paper by giving a list of open problems in Section \ref{se:6} and commenting on possible approaches towards them.

\section{Preliminaries and notation}\label{se:prel}

In the following, let $K$ be a totally real number field, $\O_K$ its ring of integers and $\U_K \subset \O_K$ its group of units. If $d$ is the degree of $K$ over $\mathbb{Q}$, then there exist $d$ distinct embeddings $\sigma_i:K\hookrightarrow \mathbb{R}$, $i=1,\ldots,d$. 
The group of automorphisms of $K$ that fix $\Q$ is denoted by $\mathrm{Aut}(K/\Q)$. 
We denote by $\NN$ and $\Tr : K \to \Q$  the norm and trace from $K$ to $\mathbb{Q}$, i.e., the product and sum of all the embeddings $\sigma_i$. 

For $\alpha,\beta\in K$, we say that $\alpha$ is totally greater than $\beta$, and denote it by $\alpha \succ \beta$ (or $\beta\prec\alpha$), if $\sigma_i(\alpha)>\sigma_i(\beta)$ for all $\sigma_i$; $\alpha\succeq\beta$ (and $\beta\preceq\alpha$) means that  $\alpha \succ \beta$ or $\alpha=\beta$. 
We write $\OPlus_K$ for the set of totally positive integers, i.e., elements $\alpha\in\O_K$ such that $\alpha\succ 0$, and $\UPlus_K$ for the group of totally positive units.

Instead of working with quadratic forms (i.e., homogeneous polynomials of degree $2$), we use the geometric language of quadratic lattices. For notations and terminology about quadratic lattices over number fields, we refer the reader to the standard reference \cite{OM}. However, we repeat some of the most important notions here.

By a \textit{lattice} $L$ we shall always mean a totally positive definite integral quadratic $\O_K$-lattice, i.e., $L$ is a finitely generated torsion-free $\O_K$-module endowed with a quadratic map $Q: L \to \O_K$ such that, for all nonzero $v\in L$, $Q(v)\in\O_K^+$. Recall that a \emph{quadratic map} is a map $Q$ such that 1) $Q(\alpha v) = \alpha^2 Q(v)$ for all $\alpha \in \O_K$, $v \in L$, and 2) the map $B(v,w)=\bigl(Q(v+w)-Q(v)-Q(w)\bigr)/2$ is bilinear. We say that the lattice $L$ is \textit{classical} if $B(v,w)\in\O_K$ for all $v,w\in L$.

We say that a lattice is \emph{free} if the underlying $\O_K$-module is free; this is always the case if $h(K)=1$, i.e., the field $K$ has class number $1$. Free lattices are in one-to-one correspondence with quadratic forms understood as polynomials: For a free lattice with basis $v_1, \ldots, v_n$, we can define the corresponding quadratic form as $\phi(x_1,\ldots, x_n) = Q(x_1v_1+\cdots + x_nv_n)$, and vice versa, $\phi$ can easily be used to define $Q$ on a free lattice. Generally, if $h(K) > 1$, there are non-free lattices. However, by the structure theorem for Dedekind domains, the underlying module of a lattice can be written as $\O_K v_1 + \cdots + \O_K v_{n-1} + \mathfrak{a}v_n$ where the vectors $v_1, \ldots, v_n$ are linearly independent (they are sometimes called a \emph{pseudobasis}, although the term is more general) and $\mathfrak{a}$ is a fractional ideal \cite[81:5]{OM}. The number of elements $n$ in the pseudobasis is an invariant of $L$, called the \emph{rank} of~$L$ \cite[81:6]{OM}. The class of $\mathfrak{a}$ in the class group of $K$ is another invariant of $L$, called its \emph{Steinitz class}. Note that, by \cite[\S 81C]{OM} (or \cite[Theorem 1.39]{Nar}), the Steinitz class of a lattice given as $L=\mathfrak{a}_1 v_1 + \cdots  + \mathfrak{a}_nv_n$ is $\mathfrak{a}_1 \cdots   \mathfrak{a}_n$, and $L$ is free if and only if $\mathfrak{a}_1 \cdots   \mathfrak{a}_n$ is a principal fractional ideal. 

Formally, one should always write a quadratic lattice as a tuple $(L,Q)$; however, all our lattices will be quadratic, and, if we do not specify otherwise, the corresponding quadratic map will always be denoted by $Q$. Two lattices are \emph{isometric} if there exists an $\O_K$-linear bijection between them which respects the quadratic map, i.e., we write $L_1 \simeq L_2$ or more precisely $(L_1,Q_1) \simeq (L_2,Q_2)$ if there exists an $\O_K$-linear bijection $\iota : L_1 \to L_2$ such that $Q_1(v) = Q_2\bigl(\iota(v)\bigr)$ for all $v \in L_1$. Typically we only study quadratic lattices up to isometry. The \emph{orthogonal sum} of $(L_1,Q_1)$ and $(L_2,Q_2)$, denoted by $L_1 \perp L_2$, is the $\O_K$-module $L_1 \oplus L_2$ (direct sum of modules) equipped with the map $Q(x_1 \oplus x_2) = Q_1(x_1) + Q_2(x_2)$.

For $a_1,a_2,\ldots,a_n\in\O_K^+$, we denote by $\qf{a_1,a_2,\ldots,a_n}$ the quadratic form $Q(x_1,\dots,x_n)=a_1x_1^2+a_2x_2^2+\dots+a_nx_n^2$, as well as  the corresponding \textit{diagonal lattice} $(\O_K^n,Q)$ (by a slight abuse of notation which should cause no confusion). 
As the term ``diagonal lattice'' is not entirely standard, let us stress that we will always use it in this sense, i.e., to mean a ``lattice corresponding to a diagonal form''. Clearly, every diagonal lattice is free and classical.

We say that the lattice $L$ \emph{represents} $\alpha\in\OPlus_K$, and write $\alpha\to L$, if there exists $v\in L$ such that $Q(v)=\alpha$. A lattice $L$ representing every element of $\OPlus_K$ is called \emph{universal}.

By $\OO$ we denote the set of classes of elements of $\O_K^+$ up to multiplication by $\U_K^2$, i.e., by squares of units.
Observe that if $\alpha\in\OPlus_K$ and $\epsilon \in\U_K$, then $\alpha\to L\Leftrightarrow\alpha\epsilon^2\to L$. Thus when talking about representations of elements by $L$, the only thing that matters is their class modulo squares of units. Therefore we will usually work with subsets $S\subset \OO$; however, we will often just write $\alpha\in S$ to mean that the class of $\alpha\in\O_K^+$ lies in $S$. For $S\subset \OO$, we define a lattice to be \textit{$S$-universal} if it represents all the elements of $S$ (i.e., more precisely, if it represents all $\alpha\in\OPlus_K$ such that the class $\alpha\U_K^2\in S$).

If $\alpha\U_K^2=\beta\U_K^2$, then $\NN(\alpha)=\NN(\beta)$, and so we can extend the norm to a well-defined map $\NN:\OO\rightarrow \Z^+$. Also, for $\sigma\in\mathrm{Aut}(K/\Q)$ we define $\sigma:\OO\rightarrow\OO$ as $\sigma(\alpha\U_K^2)=\sigma(\alpha)\U_K^2$.

A crucial tool in studying representations are local representations, which we are now going to discuss briefly. For a totally real number field, there are two types of places: finite places $\p$ corresponding to prime ideals, and infinite places corresponding to embeddings into $\R$. We are going to mostly ignore the latter; it suffices to say (without giving the definitions) that since our lattices are totally positive definite, a nonzero lattice is universal at all infinite places.

For a finite place $\p$, we write $\O_\p$ for the completion of $\O_K$ at $\p$ (the ring of $\p$-adic integers), and $K_\p$ for the corresponding field. One defines $\O_\p$-lattices analogously to $\O_K$-lattices, only omitting the \enquote{positive definite} part of the definition, which makes no sense for the ring $\O_\p$. To localize a lattice, we use the \textit{tensor product} -- for an $\O_K$-lattice $(L=\O_K v_1 + \cdots + \O_K v_{n-1} + \mathfrak{a}v_n,Q)$, it is the $\O_\p$-lattice $(L \otimes \O_\p=\O_\p v_1 + \cdots + \O_\p v_{n-1} + \mathfrak{a}\O_\p v_n,Q')$ where $Q'$ is the quadratic map satisfying $Q'(v_i)=Q(v_i)$ for all $i$. Similarly, for number fields $H \supset K$, we have the tensor product lattice $L \otimes \O_H=\O_H v_1 + \cdots + \O_H v_{n-1} + \mathfrak{a}\O_H v_n$.

We say that a lattice $L$ \textit{represents $\alpha \in \O_K$ at $\p$} if $L \otimes \O_\p$ represents $\alpha$; here, representation for $\O_\p$-lattices is defined analogously as for $\O_K$-lattices. We say that $L$ is \emph{universal at $\p$} if the lattice $L \otimes \O_{\p}$ represents all elements of $\O_{\p}$. We will also need to speak about local $S$-universality where $S$ is a subset of $\O_K$ or of $\OO$, respectively. In that case, we can naturally interpret $S$ as a subset of $\O_{\p}$ or of $\O_{\p} / (\O_{\p}^{\times})^2$, respectively. Thus, we say that $L$ \emph{is $S$-universal at $\p$} if $L \otimes \O_{\p}$ represents all elements of $S$.

For an $\O_\p$-lattice $L$, $\mathfrak{s}(L)$ is the fractional ideal in $\O_\p$ generated by all elements $B(v,w)$, $v,w \in L$. (If $\p$ is non-dyadic, i.e., $\p$ does not divide $2$, then it is an integral ideal.) We say that $L$ is \emph{unimodular} at $\p$ if $\mathfrak{s}(L) = \O_\p$. A crucial result \cite[92:1b]{OM} states that if $\p$ is non-dyadic, $L$ has rank at least 3, and $L$ is unimodular at $\p$, then $L$ is universal at $\p$. Also, every nonzero lattice is unimodular at all but finitely many places: For free lattices, the only problematic places are those that divide the discriminant; generally, one uses the volume ideal of $L$, but we shall not need this explicitly.

\subsection{New notation}

In order to give a unified treatment to some of our results pertaining to classical or diagonal lattices, we proceed as follows: For $X\in\{\nc,\cl,\diag\}$, by an \textit{X-lattice} we mean:
\begin{itemize}
    \item a lattice if $X=\nc$ ($\nc$ here stands for ``non-classical'', i.e., not necessarily classical),
    \item a classical lattice if $X=\cl$,
    \item a diagonal lattice if $X=\diag$.
\end{itemize}

\begin{definition}\label{de:crit}
Let $S \subset \OO$ and $X\in\{\nc,\cl,\diag\}$. We say that $\CC \subset S$ is an \emph{$(S,X)$-criterion set} if for every $X$-lattice $L$ over $K$, we have the following equivalence: $L$ is $S$-universal if and only if it is $\CC$-universal.
\end{definition}

If $S=\OO$, we simply speak about a \emph{universality criterion set}. For all three choices of $X$, note that $S$ itself is an $(S,X)$-criterion set. In Section \ref{se:weaker}, we show that one can in fact remove the condition $\CC \subset S$ from Definition \ref{de:crit} and still obtain virtually the same uniqueness result as in our main theorems.

Throughout most of the paper, we consider a fixed totally real number field $K$ and often also a fixed set $S \subset \OO$ and a fixed choice of $X\in\{\nc,\cl,\diag\}$.

\begin{definition}\label{de:truant}
Given a lattice $L$, we say $\alpha\in S$ is an \emph{$S$-truant} of $L$ if $\alpha \not \to L$ and $\beta\to L$ for all $\beta\in S$, $\NN(\beta)<\NN(\alpha)$. 
\end{definition}

Note that this does not always give a unique $S$-truant of $L$ as an element of $\OO$; for example, if $S = \OO$ then every totally positive unit is an $S$-truant of the zero lattice $\{{0}\}$. 

We introduce a set of numbers which obviously belong to every $(S,X)$-criterion set.

\begin{definition} \label{de:critical}
We say that $\alpha \in S$ is \textit{$(S,X)$-critical} if there exists an $X$-lattice $L$ such that:
 \begin{itemize}
     \item $\alpha \not\to L$;
     \item $L$ is $\bigl(S \setminus \{\alpha\}\bigr)$-universal.
 \end{itemize}
\end{definition}

Clearly, an $(S,X)$-critical element must belong to every $(S,X)$-criterion set. In fact, Definition \ref{de:critical} basically says that an element $\alpha$ is critical if it is possible to find a lattice representing all other elements in $S$, except $\alpha$ (and, of course, its multiples by squares of units). In this way, representing all other elements is not a sufficient condition to represent $\alpha$, and thus $\alpha$ must be included in the criterion set. 

The main aim of the next section is to show that, on the other hand, no non-critical element is needed in a criterion set, therefore providing an effective characterization of criterion sets together with the uniqueness result. 

In order to shorten the terminology, in the case $X=\nc$ we shall usually omit it and talk just about a lattice, $S$-critical element, etc. Similarly, when $S=\OO$, then we omit it, and so we then talk about universal $X$-lattice, $X$-critical element, truant, \dots, or just universal lattice and critical element if also $X=\nc$.

\section{Uniqueness} \label{se:unique}

In this section we prove Theorem \ref{th:unique}, which asserts that the set $\CC_S^X$ of all critical elements forms a criterion set. This means that $\CC_S^X$ is indeed the unique minimal criterion set, as critical elements must obviously belong to every criterion set.

The main ingredient for the proof is the following result, whose idea is that if there exists a lattice representing all elements of norm smaller than a given $\alpha$, but not $\alpha$, then it is possible to
``orthogonally escalate'' the lattice by adding vectors that represent
elements of norm greater than $\alpha$, until eventually we obtain a lattice that represents all of them. 

The proposition is interesting on its own, since it can be used to effectively determine for a given $\alpha \in S$ whether it is $(S,X)$-critical or not (see Remark \ref{rem:34}).

\begin{proposition} \label{pr:critical}
Let $\alpha \in S$. The following are equivalent:
 \begin{enumerate}
     \item[(a)] $\alpha$ is $(S,X)$-critical;
     \item[(b)] there exists an $X$-lattice $L$ for which $\alpha$ is an $S$-truant. 
 \end{enumerate}
\end{proposition}
\begin{proof}
The implication (a) $\Rightarrow$ (b) is trivial. Indeed, by (a) there exists an $X$-lattice $L$ representing all elements of $S$ except $\alpha$, which in particular represents all elements of $S$ of smaller norm than $\alpha$, i.e., $\alpha$ is an $S$-truant of $L$.
For the other implication, let us first prove it for $X=\nc$ and then explain at the end that the proof remains unchanged when $X=\cl, \diag$.

Thus assume that the lattice $L$ has $S$-truant $\alpha$; we shall construct another lattice $L'$ which proves that $\alpha$ is $S$-critical.

Put $L_0 = L$ and then, as long as $L_i$ is not $(S \setminus \{\alpha\})$-universal, choose $\beta_{i+1}$ to be an $(S \setminus \{\alpha\})$-truant of $L_i$, and put $L_{i+1} = L_i \perp \qf{\beta_{i+1}}$.

First we show by induction that none of the thus produced lattices $L_i$ represent $\alpha$. It is not represented by $L_0$. Assume now that $\alpha$ is not represented by $L_i$; we aim to show that then $\alpha$ is not represented by $L_{i+1}$ either. Since already $L_0$ represented all elements of norm strictly smaller than $\alpha$, the same holds for $L_i$, hence $\NN(\beta_{i+1}) \geq \NN(\alpha)$. Each element represented by $L_{i+1}$ can be written as $\gamma + \omega^2 \beta_{i+1}$ where $\gamma \to L_i$. For the sake of contradiction, assume $\alpha \to L_{i+1}$. Then $\alpha = \gamma + \omega^2 \beta_{i+1}$ where $\omega \neq 0$, as $\alpha \not\to L_i$. Thus
\[
\NN(\alpha)=\NN(\gamma + \omega^2 \beta_{i+1}) \geq \NN(\omega^2 \beta_{i+1}) \geq \NN(\beta_{i+1}) \geq \NN(\alpha);
\]
this means that all the inequalities must in fact be equalities. The first one is an equality if and only if $\gamma = 0$, the second if and only if $\omega$ is a unit. Thus our assumptions $\alpha \to L_{i+1}$ but $\alpha \not\to L_i$ lead to $\alpha \in \U_K^2 \beta_{i+1}$. However, this is a contradiction, since in this case $\alpha$ and $\beta_{i+1}$ are the same element of $\OO$, but we chose $\beta_{i+1}\in S \setminus \{\alpha\}$. 

We showed that none of $L_i$ represent $\alpha$. Thus, if we further show that one of these lattices is $(S \setminus \{\alpha\})$-universal, we have found the desired lattice which proves that condition (a) holds.

To achieve this, first note that every element of $S \setminus \{\alpha\}$ is represented by all but finitely many of the $L_i$. This is an easy consequence of the fact that there are only finitely many elements of $\OO$ under a given norm; indeed, at each step, either the norm of the truants of $L_{i+1}$ is larger than for $L_i$, or $L_{i+1}$ has fewer truants than $L_i$.

Our next aim is proving that from some index $N'$ onwards, the lattices $L_i$, $i \geq N'$, are locally everywhere $(S \setminus \{\alpha\})$-universal. After possibly replacing $L_0$ by $L_3$, we can assume that $L_0$ has rank at least $3$. (It can happen that there is no lattice $L_3$, since our process stops earlier; however, that happens exactly if $L_0$, $L_1$ or $L_2$ is $(S \setminus \{\alpha\})$-universal and there is nothing to prove.)

Denote by $P = \{\p_1, \ldots, \p_k\}$ the finitely many places which are either dyadic or at which $L_0$ is not unimodular. It is well known \cite[92:1b]{OM} that $L_0$ is universal at all places outside of $P$ (as rank of $L_0$ is at least 3). Now, for a place $\p_j$ we want to find an index $N_j$ such that every $L_i$, $i \geq N_j$, is $(S \setminus \{\alpha\})$-universal at $\p_j$. Once such $N_j$ was found, it is enough to put $N' = \max\{N_1, \ldots, N_k\}$.

Fix the place $\p_j$ and denote by $T$ the set of all square classes in the field $K_{\p_j}$ that have nonempty intersection with $S \setminus \{\alpha\}$. If $\p_j$ is non-dyadic, then $\# T \leq 4$ and the classes are represented by some of $1, \epsilon, \pi, \epsilon\pi$ where $\pi$ is a prime and $\epsilon$ a unit in $\O_{\p_j}$; in general, $T = \{t_1, \ldots, t_n\}$ is a finite set. From every class $t_i$, pick a representative $\tau_i \in (S \setminus \{\alpha\}) \cap t_i$ such that no element of $(S \setminus \{\alpha\}) \cap t_i$ has smaller $\p_j$-valuation. Then the lattice $\qf{\tau_i}$ represents at $\p_j$ all elements of $(S \setminus \{\alpha\}) \cap t_i$. Thus, every lattice representing all elements $\tau_1, \ldots, \tau_n$ is $(S \setminus \{\alpha\})$-universal at $\p_j$. We denote the first lattice representing all of them as $L_{N_j}$ and we are done with this part of the proof.

Let us now consider the lattice $L_{N'}$ and assume without loss of generality that it has rank at least $5$. Since it is locally everywhere $(S \setminus \{\alpha\})$-universal, by \cite{HKK} there are only finitely many elements of $(S \setminus \{\alpha\})$ which are not represented by $L_{N'}$. After taking a few more steps, they will all be represented and the lattice $L'=L_{N'+n}$ will be $(S \setminus \{\alpha\})$-universal, as claimed.

Concerning the cases $X=\cl$ and $X=\diag$, note that we have constructed the lattice $L'$ proving $S$-criticality of $\alpha$ as  some lattice $L_i=L\perp\qf{\beta_1,\dots,\beta_i}$. Thus if $L$ is an $X$-lattice, then so is $L'$, and the preceding argument still shows that $L'$ is $(S \setminus \{\alpha\})$-universal, as needed.
\end{proof}

Note that Proposition \ref{pr:critical} does not hold in the setting of $k$-universality, i.e., when we want to represent quadratic  lattices of a given rank $k$. First of all, there is no natural way of ordering the rank $k$ lattices in order to define truants. One could (non-uniquely) choose such an ordering -- but then the proof would fail anyway, for it could definitely happen that when constructing the chain of lattices $L_i$, some of these lattices may represent the ``truant'' $\alpha$ anyway. From the other direction, we know that criterion sets for $k$-universality are not unique for $k\geq 9$ \cite{KLO}. The reason is that in rank 8, there exists the indecomposable unimodular $\Z$-lattice $E_8$. Some lattice of the form $E_8\perp L_0$ must lie in every minimal criterion set (for $k=9$, say), and one can show that it can be replaced by $E_8\perp L$ for any other lattice $L$.

\medskip

We can now easily prove that the set of critical elements is a criterion set. First, let us officially introduce the key notation.

\begin{definition} \label{def:C_S}
Let $K$ be a totally real number field, $S \subset \OO$, and $X\in\{\nc,\cl,\diag\}$.

We denote the set of all $(S,X)$-critical elements as $\CC_S^X$.

When $S=\OO$, we write $\CC_K^X$ instead of $\CC_{\OO}^X$.    

When $X=\nc$, then we write $\CC_S$ or $\CC_K$ instead of $\CC_S^{\nc}$ or $\CC_K^{\nc}$.
\end{definition}

\begin{theorem} \label{th:unique} Let $K$ be a totally real number field, $S \subset \OO$, and $X\in\{\nc,\cl,\diag\}$. Then:

(a) $\CC_S^X$ is an $(S,X)$-criterion set. 

(b) $\CC \subset S$ is an $(S,X)$-criterion set if and only if $\CC_S^X \subset \CC$.

(c) $\Cdiag_S \subset \Ccl_S \subset \CC_S.$

(d) $\CC_S^X$ is finite.
\end{theorem}

\begin{proof}
(a) Let $L$ be an $X$-lattice which is $\CC_S^X$-universal. If it is not $S$-universal, there is an $S$-truant $\alpha$ of $L$. However, by Proposition \ref{pr:critical}, $\alpha \in \CC_S^X$ and thus $\alpha \to L$, which is a contradiction. Hence $L$ is $S$-universal, and thus $\CC_S^X$ is indeed an $(S,X)$-criterion set.

(b) Clearly a set containing $\CC_S^X$ is also an $(S,X)$-criterion set; on the other hand, we already observed that an $(S,X)$-criterion set must necessarily contain all $(S,X)$-critical elements.

(c) $\CC_S$ is also an $(S,\cl)$-criterion set, and so $\Ccl_S \subset \CC_S$ by (b). The first inclusion $\Cdiag_S \subset \Ccl_S$  follows in the same way.

(d) Chan--Oh \cite[Theorem 5.7]{CO} showed that there is a finite $S$-criterion set $S_0$, and so our minimal $S$-criterion set $\CC_S$ is finite, as it is contained in $S_0$ by (b). The inclusions in (c) then imply also the finiteness of $\Cdiag_S$ and $\Ccl_S$.
\end{proof}

Hence we now know that, \emph{for all subsets $S$ of $\OO$, there is a unique minimal $S$-criterion set which is finite}. 

\begin{remark}\label{rem:34} As there are finitely many elements of $\OO$ of bounded norm,
    checking whether an element $\alpha$ is $(S,X)$-critical is a finite procedure: One needs to check the representability of the finitely many elements of smaller norm by the finitely many lattices that may be relevant for the criticality of $\alpha$ (both can be best seen using escalations as explained in the follow-up paper \cite{KR}; or one can use bounds such as \cite[Lemma 4]{EK}). In principle, one can thus (provably!) find all the elements of $\CC_S^X$ whose norms are smaller than a given bound.
\end{remark}

Let us now note several consequences in the key case where $S$ contains all totally positive integers. Therefore, until the end of the section, let us consider $S=\OO$.

\begin{corollary} \label{co:closed}
The set $\CC_K^X$ is closed both under conjugation, i.e., for each $\sigma\in\mathrm{Aut}(K/\Q)$ we have $\sigma(\CC_K^X)=\CC_K^X$, and under multiplication by totally positive units.
\end{corollary}
\begin{proof}
Let $\sigma\in \mathrm{Aut}(K/\Q)$ and $\epsilon \in \UPlus_K$. In fact, we see something more general: Whenever $(L,Q)$ is $\bigl(S \setminus \{\alpha\}\bigr)$-universal for some $S$, then $(L,\epsilon Q)$ is $\bigl(\epsilon S \setminus \{\epsilon\alpha\}\bigr)$-universal. Similarly, the lattice $\bigl(L,\sigma(Q)\bigr)$, where the map $\sigma(Q)$ is defined by $v \mapsto \sigma\bigl(Q(v)\bigr)$, is $\bigl(\sigma(S) \setminus \{\sigma(\alpha)\}\bigr)$-universal. Thus the result holds since for $S = \OO$ we have $\sigma(S)=S=\epsilon S$.
\end{proof}

It is important to stress that the statement $\sigma(\CC_K^X) = \CC_K^X$ really requires $\CC_K^X$ to be understood as a subset of $\OO$. For example, it is imprecise to formulate Lee's theorem \cite{Le} as $\Ccl_{\Q(\sqrt{5})} = \{1,2,\frac{5+\sqrt5}{2}, \frac{7+\sqrt5}{2}, \frac{7-\sqrt5}{2}, 2\frac{5+\sqrt5}{2}, 3\frac{5+\sqrt5}{2}\}$. Normally, this sloppiness does not cause any confusion, but here one must see that Corollary \ref{co:closed} is not violated since $\frac{5-\sqrt5}{2} \in \frac{5+\sqrt5}{2} \U_K^2$.

Since real quadratic fields with totally positive fundamental unit have density one (when ordered by discriminant; for much more information on this topic see, e.g., \cite{St}), Corollary \ref{co:closed} has a funny consequence.

\begin{corollary}\label{co:density}
For density $1$ of real quadratic fields $K$, $\#\CC_K,\#\Ccl_K$, and $\#\Cdiag_K$ are all even.
\end{corollary}

We also obtain the following cute corollary.

\begin{corollary}\label{pr:quadratic}
(a) Denote by $\mdiag_K$ the minimal rank of a universal diagonal form. Then 
$$\mdiag_K\leq \#\Cdiag_K$$

(b) There are only finitely many real quadratic fields $K$ where $\#\Cdiag_K \leq 7$.
\end{corollary}
\begin{proof}
(a) If $\Cdiag_K = \{\alpha_1,\ldots,\alpha_n\}$, then clearly the diagonal form $\qf{\alpha_1, \ldots, \alpha_n}$ is universal.

(b) This is obvious from part (a), as Kim--Kim--Park \cite{KKP} showed that there are only finitely many real quadratic fields admitting a universal quadratic lattice in less than $8$ variables. 
\end{proof}

Note that the diagonal statement in part (a) above is stronger than, e.g., the analogous observation that $m_K \leq \#\CC_K$, for we clearly have $m_K\leq \mdiag_K\leq \#\Cdiag_K\leq \#\CC_K$.

\section{Indecomposable elements, rational integers} \label{se:indec}

Let us continue focusing on the case $S=\OO$ of universality criterion sets and produce some elements that must be contained in them. 
By Theorem \ref{th:unique}(c), we have $\Cdiag_S \subset \Ccl_S \subset \CC_S$, and so we will usually work with $\Cdiag_K$, as this will imply the same results for the larger criterion sets.  This section analyzes specific elements of $\OPlus_K$ for which we can determine that they must lie in the minimal criterion set. Further down in the text this will be studied for rational integers, but first we show that indecomposable elements of $\O_K$ mostly belong to the minimal criterion set. Also, for quadratic fields $K$, we use the more precise knowledge of the additive structure of $\OPlus_K$ to explicitly find a quadratic form representing everything but a given squarefree indecomposable element.

We say that $\gamma \in \O_K$ (or $\gamma\in\OO$) is \emph{squarefree} if it is not divisible by any nontrivial square in the sense that $\omega^2 \mid \gamma$ only holds for $\omega \in \U_K$. Note that this is a condition about divisibility by elements, not by ideals. We start with a simple observation which generalizes the fact that the minimal criterion set over $\Z$ contains only squarefree integers.

\begin{lemma} \label{ob:squarefree}
The criterion sets $\Cdiag_K\subset  \Ccl_K\subset \CC_K$ only contain squarefree elements.
\end{lemma}
\begin{proof}
If a critical element satisfies $\beta = \alpha \omega^2$ for $\alpha,\omega \in \O_K$, then $\alpha \succ 0$, and clearly every lattice representing $\alpha \in \OPlus_K$ represents $\beta$ as well. This is a contradiction unless $\omega$ is a unit.
\end{proof}

A simple yet powerful theorem is that for indecomposable elements, this is in fact an equivalence.

\begin{theorem} \label{th:indecomposable}
Let $\beta \in \OPlus_K$ be a squarefree indecomposable element. Then $\beta \in \Cdiag_K$. 
\end{theorem}
\begin{proof}

By Theorem \ref{th:unique}, we have to show that $\beta$ is critical, and by Proposition \ref{pr:critical} it is enough to find a diagonal lattice $L$ which represents every $\alpha \in \OPlus_K$ for which $\NN(\alpha) < \NN(\beta)$. There are only finitely many classes of such numbers in $\OO$; denote their representatives as $\alpha_1, \ldots, \alpha_k$. Then $L=\qf{\alpha_1, \ldots, \alpha_k}$ satisfies the assumptions: Thanks to indecomposability, $\beta \to L$ would mean $\beta \to \qf{\alpha_i}$ for some $i$, which is equivalent to $\beta = \alpha_i \omega^2$ for some $\omega$. This is a contradiction, so $L$ indeed has truant $\beta$.
\end{proof}

In light of Theorem \ref{th:indecomposable}, every squarefree indecomposable element $\beta$ is critical and there exists a lattice $L$ representing everything except $\beta$. However, in practice it is difficult to exhibit such a lattice and prove that it has this property. The reason is that the constructions based on Proposition \ref{pr:critical} rely on the asymptotic results of \cite{HKK}, which give notoriously large upper bounds. For quadratic fields, we provide an alternative construction which directly yields such a lattice $L$ for every squarefree indecomposable element.

\begin{example} \label{ex:explicitForm}
\emph{Let $K$ be a real quadratic field and $\beta\in\OPlus_K$ a squarefree indecomposable element. We shall construct explicitly a diagonal quadratic form over $\O_K$ which represents all totally positive integers except for $\beta\epsilon^2$, $\epsilon \in \U_K$.}

In accordance with \cite{HK}, we order the indecomposables by their absolute value to get the bi-infinite sequence $\ldots, \beta_{-1}, \beta_0=1, \beta_1, \ldots$ where $\beta_i < \beta_{i+1}$ for all indices, and denote by $t$ the index for which $\beta_t$ is the square of the fundamental unit. We have $\beta = \beta_{k+nt} = \epsilon_{\mathrm{fund}}^{2n}\beta_k$ for a unique $0 \leq k \leq t-1$ and $n\in\Z$. We claim that (unless $t=1$, which only happens for $\Q(\sqrt5)$ and has to be handled separately) the diagonal form $D$ defined as
\[
D = {}^{\beta_0\!}I_4 \perp \ldots \perp {}^{\beta_{k-1}\!}I_4 \perp {}^{\beta_{k+1}\!}I_4 \perp \ldots \perp {}^{\beta_{t-1}\!}I_4 \perp {}^{\beta_k\!}J \perp \qf{\beta_{k-1}+\beta_k} \perp \qf{\beta_k+\beta_{k+1}},
\]
has the desired property; here $J =\qf{2,2,3,4}$ denotes a form which over $\Z$ represents all natural numbers except for $1$, and $I_4 =\qf{1,1,1,1}$, and the upper left index denotes scaling of the quadratic form, e.g., ${}^{\beta_0\!}I_4 = \qf{\beta_0,\beta_0,\beta_0,\beta_0}$ and ${}^{\beta_k\!}J=\qf{2\beta_k,2\beta_k,3\beta_k,4\beta_k}$.

First of all, we observe that $D$ does not represent $\beta_k$. As $\beta_k$ is indecomposable, it is enough to prove that it is not represented by each of the orthogonal summands. The summands $\qf{2\beta_k}$, $\qf{3\beta_k}$, $\qf{4\beta_k}$, $\qf{\beta_{k-1}+\beta_k}$, and $\qf{\beta_k+\beta_{k+1}}$ do not represent any indecomposable elements whatsoever. If $\beta_k \to \qf{\beta_i}$, then by the squarefree property we have $\beta_k = \epsilon_{\mathrm{fund}}^{2m}\beta_i$ for some $m$, but this is impossible for $i \neq k$, $0 \leq i \leq t-1$. Thus $\beta_k$ is not represented by $D$, and neither is $\beta$.

Assume now that $\alpha \in \OPlus_K$ and $\alpha \not\in \beta_k\U^2_K$. We shall prove that it is represented by $D$. By \cite[Theorem~2]{HK}, one can write $\alpha = m\beta_j + n\beta_{j+1}$ for nonnegative $m,n \in \Z$. If necessary, multiply $\alpha$ by a suitable power of $\epsilon_{\mathrm{fund}}^2$ so that $0 \leq j \leq t-1$. Since $I_4$ represents every natural number, then $\alpha \to D$ unless $k = j$ or $j+1$. Thanks to the summand ${}^{\beta_k\!}J$, the case when the coefficient in front of $\beta_k$ is greater than $1$ is also trivial. Finally, if $\alpha = \beta_k + m'\beta_{k\pm1}$, we can make use of the summand $\qf{\beta_k+\beta_{k\pm1}}$ (unless $m'=0$, which is exactly the case $\alpha = \beta_k$ when we claim $\alpha \not\to D$).

When $K = \Q(\sqrt5)$, there is only one indecomposable element, namely $1$, so $\beta_i = \varphi^{2i}$ where the golden ratio $\varphi = \frac{1+\sqrt{5}}{2}$ is the fundamental unit. Up to multiplication by $\U_K^2$, every element can be written as $m + n\varphi^2$. The lattice which represents everything but $1$ can for example be constructed as $D = J \perp J \perp \qf{1+\varphi^2} \perp \qf{2+\varphi^2} \perp \qf{1+2\varphi^2}$; there are also nicer but slightly less elementary choices.
\end{example}

Now we shall switch our attention from indecomposable elements to rational integers. Note that the first claim of the following statement is clear from the preceding discussion, as $1$ is always a squarefree indecomposable element.

\begin{proposition} \label{pr:123}
Let $K$ be a totally real number field. Then:
 \begin{enumerate}
 \item[(a)] $1 \in \Cdiag_K$;
 \item[(b)] $2 \in \CC_K$ if and only if $2 \in \Cdiag_K$, which happens if and only if $2$ is squarefree in $K$;
 \item[(c)] if both $2$ and $3$ are squarefree in $K$ and $\sqrt5 \notin K$, then $3 \in \Cdiag_K$.
 \end{enumerate}
\end{proposition}
\begin{proof}
In all parts, we use Proposition \ref{pr:critical}. The claim about $1$ is trivial, as $\{0\}$ has truant $1$. Further, we use the fact that the only decompositions of $2$ and $3$ as a sum of two or more totally positive algebraic integers are $2=1+1$ and $3=2+1=1+1+1=\bigl(\frac{1+\sqrt5}{2}\bigr)^2+\bigl(\frac{1-\sqrt5}{2}\bigr)^2$. This follows from known results on the Schur--Siegel--Smythe trace problem \cite[Theorem III]{Si1}: in such a decomposition, at least one of the summands must have absolute trace at most $3/2$, and the only such elements are $1$ and $\frac{3\pm\sqrt5}{2}$. For an explicit proof concerning the decompositions of $2$ and $3$, see \cite[Lemma 2.2]{KKK}.

Clearly $2 \in \Cdiag_K$ implies $2 \in \CC_K$, which in turn implies that $2$ is squarefree. We close the circle of implications by constructing a diagonal lattice $D$ with truant $2$, assuming that $2$ is squarefree. Consider all squarefree elements of $\OO$ with norm under $\NN(2)$ and let $\alpha_1,\ldots,\alpha_n$ be their representatives. Then $D = \qf{\alpha_1,\ldots,\alpha_n}$ is a suitable lattice; it fails to represent $2$, since $2 \not\to \qf{\alpha_i}$ for every $i$ by the squarefree property of $2$, and $1$ is represented only by one of the forms $\qf{\alpha_i}$, not by two of them.

For the last statement, we construct a diagonal lattice with truant $3$. This time, let $\alpha_1,\ldots,\alpha_n$ be representatives of all squarefree elements of $\OO$ with norm under $\NN(3)$ with the exclusion of $\U_K^2$ and $2\U_K^2$. Then the lattice $D=\qf{1,1,\alpha_1,\ldots,\alpha_n}$ represents all elements with norm under $\NN(3)$, and it remains to show that it does not represent $3$. Since $3$ is squarefree, it is represented neither by $\qf{1}$ nor by any $\qf{\alpha_i}$. Thus, any representation should use some decomposition of $3$. However, since $2$ is squarefree, it is also represented neither by $\qf{1}$ nor by any $\qf{\alpha_i}$. Therefore, for fields $K$ not containing $\sqrt5$, the only option is $1+1+1$. However, none of the $\qf{\alpha_i}$ represent $1$. Thus $3 \not\to D$, which concludes the proof.
\end{proof}

\noindent
\textbf{Convention.} In the rest of the section, we will establish that, under certain conditions, some elements of the criterion set for a field $F$ must lie in the criterion set for another field $K$. We will formulate the statements as ``$\CC_{F}^X\cap A \subset \CC_K^X$'' (for various sets $A\subset \O_F^+$). However, this is formulated imprecisely, as criterion sets are subsets of $\OO$, not $\OPlus_K$ (recall Definition \ref{de:crit}). The precise versions would be: \textit{For each $a \in \O_F^+,  a\in A$, we have $a\U_F^2\in \CC_F^X \Rightarrow a \U_K^2\in \CC_K^X$.}

\begin{proposition} \label{pr:15iscritical}
Let $K$ be a totally real number field and let $n \in \Cdiag_{\Q}=\{1,2,3,5,6,7,10,14,15\}$.
Assume that:
 \begin{enumerate}
     \item[(a)] If $\alpha \preceq n$ for $\alpha \in \OPlus_K$, then $\alpha \in \Z$.
     \item[(b)] For all $m \in \{1,2,\dots, n\}$ and $\alpha,\omega\in\O_K$ such that $m=\alpha\omega^2$, there exists $\epsilon \in \U_K$ satisfying $\epsilon\omega \in \Z$.
 \end{enumerate}
Then
\[
\Cdiag_{\Q}\cap \{1,2,\dots, n\} \subset \Cdiag_K.
\]
\end{proposition}

Note that condition (b) for $m=1$ is always satisfied, as one can take $\epsilon=\omega^{-1}$. Also note that $\Cdiag_{\Q}=\Ccl_{\Q}$ \cite{Bh}.

\begin{proof}
It suffices to prove that:

\textbf{Claim.} Under the assumptions of Proposition \ref{pr:15iscritical}, we have
$n\in \Cdiag_K$.

For when we take $k\in \Cdiag_{\Q}\cap \{1,2,\dots, n\}$, we can then apply the Claim to the element $k$ in place of $n$ (the assumptions (a) and (b) for $k$ are a direct weakening of these assumptions for $n$, as $k\preceq n$) to obtain that also $k\in \Cdiag_K$.

\medskip

To prove the Claim, we shall construct a diagonal $\O_K$-lattice with truant $n$. To prevent confusion, we will denote diagonal $\Z$-lattices as $\qf{\dots}_\Z$ and diagonal $\O_K$-lattices as $\qf{\dots}$.
As $n\in \Cdiag_\Q$, there is a diagonal $\Z$-lattice $T_n$ with truant $n$. While we will not need it, we in fact know these lattices explicitly, e.g., we can take $T_1=\qf{}_\Z = \{0\}$, $T_2=\qf{1}_\Z$, $T_3=\qf{1,1}_\Z$, $T_5=\qf{1,2}_\Z$, $T_6=\qf{1,1,3}_\Z$, $T_7=\qf{1,1,1}_\Z$, $T_{10}=\qf{1,2,3}_\Z$, $T_{14}=\qf{1,1,2}_\Z$, and $T_{15}=\qf{1,2,5,5}_\Z$. Let $T=T_n=\qf{a_1,\dots,a_\ell}_\Z$.

Consider all squarefree elements of $\OO$ of norm less than $\NN(n)$ such that
their class does not equal $c\U_K^2$ for $c \in \Z^+$. 
Denote their representatives as $\alpha_1, \ldots, \alpha_k\in\O_K^+$ and define the diagonal $\O_K$-lattice $D=\qf{\alpha_1,\ldots,\alpha_k}$ (note that it can happen that there are no such elements, in which case $k=0$ and $D$ is empty).

Now we claim that the diagonal $\O_K$-lattice 
$$L = D\perp T\otimes\O_K= \qf{\alpha_1,\ldots,\alpha_k}\perp \qf{a_1,\dots,a_\ell}$$ 
represents all totally positive elements of norm less then $\NN(n)$: It clearly suffices to check this for squarefree elements, and those are represented either by $T \otimes \O_K$ if they are a $\U_K^2$-multiple of an element of $\Z^+$, or by $D$ if they are not.

To show that $n$ is a truant of $L$ (and thus the Claim and hence the whole Proposition holds), it remains to show that $L$ does not represent $n$. For contradiction, assume that it does, i.e., 
$$n=\sum_{i=1}^k\alpha_ix_i^2+\sum_{j=1}^\ell a_jy_j^2, \text{ where } x_i,y_j\in\O_K.$$ 

By assumption (a), we have $\alpha_ix_i^2, a_jy_j^2\in\{0,1,2,\dots,n\}$. By assumption (b) we then have 
$$x_i=\epsilon_it_i,y_j=\zeta_ju_j \text{ with } \epsilon_i,\zeta_j\in \U_K,t_i,u_j\in\Z.$$
(While we cannot apply (b) when $x_i=0$ or $y_j=0$, in such a case we can set $t_i=0$ or $u_j=0$ to reach the same conclusion; it will be also convenient to set $\zeta_j=1$ in the latter case).

If $t_i\neq 0$ for some $i$, then we have $\alpha_i\epsilon_i^2 = \alpha_i x_i^2 / t_i^2 \in \Q \cap \O_K = \Z$ with $\epsilon_i\in \U_K$, contradicting our choice of $\alpha_i$. Thus $t_i=0$ for all $i$.

We now conclude the proof by showing $\zeta_j^2 = 1$ for all $j$. When $u_j = 0$, this is immediate by our choice $\zeta_j=1$. If $u_j\neq 0$ for some $j$, then we have $a_ju_j^2\zeta_j^2\in\Z$ with $a_j,u_j\in \Z$. Thus $\zeta_j^2\in\Z$, but as $\zeta_j\in \U_K$ is invertible, we must have $\zeta_j^2=\pm 1$, and so indeed $\zeta_j^2=1$. 
We thus have 
$$n=\sum_{j=1}^\ell a_ju_j^2,$$
which means that $T$ represents $n$ over $\Z$, a contradiction.
\end{proof}

With a bit extra work we can establish that, under suitable conditions, the elements of 
\begin{align*}
    \CC_{\Q}=&\{1, 2, 3, 5, 6, 7, 10, 13, 14, 15, 17, 19, 21, 22, 23, 26,\\
&29, 30, 31, 34, 35, 37, 42, 58, 93, 110, 145, 203, 290\}
\end{align*}
are also contained in criterion sets.

\begin{proposition} \label{pr:290iscritical}
Let $K$ be a totally real number field and take $n \in \CC_{\Q}$. 
Assume that:
 \begin{enumerate}
     \item[(a)] If $\alpha \preceq 2n$ for $\alpha \in \OPlus_K$, then $\alpha \in \Z$.
     \item[(b)] For all $m \in \{1,2,\dots, 2n\}$ and $\alpha,\omega\in\O_K$ such that $m=\alpha\omega^2$, there exists $\epsilon \in \U_K$ satisfying $\epsilon\omega \in \Z$.
 \end{enumerate}
Then
\[
\CC_{\Q}\cap \{1,2,\dots, n\} \subset \CC_K.
\]
\end{proposition}

\begin{proof}
We proceed similarly as in the proof of Proposition \ref{pr:15iscritical}; again it suffices to show that $n\in\CC_K$, for which we shall construct an $\O_K$-lattice with truant $n$. 

As $n\in \CC_\Q$, there is a quadratic $\Z$-form $T$ with truant $n$. This time we will need some information about $T_n$, specifically the fact that $T_n$ can be chosen of rank $\leq 5$, see \cite[proof of Theorem 2]{BH}.
The quadratic form $2T$ is then classical and has rank $\leq 5$, so we know that $2T$ is represented by the sum of squares of integer linear forms by the theorem of Mordell and Ko \cite{Ko}, i.e.,
\begin{equation}\label{eq:2T}
    2T(Y_1,\dots,Y_\ell)=\sum_{h=1}^g L_h(Y_1,\dots,Y_\ell)^2, \text{ where } L_h(Y_1,\dots,Y_\ell)=\sum_{j=1}^\ell a_{hj} Y_j, a_{hj}\in\Z.
\end{equation}

Consider all squarefree elements of $\OO$ of norm less than $\NN(n)$ 
such that
their class does not equal $c\U_K^2$ for $c \in \Z^+$. 
Denote their representatives as $\alpha_1, \ldots, \alpha_k\in\O_K^+$ and define the diagonal $\O_K$-lattice $D=\qf{\alpha_1,\ldots,\alpha_k}$ (again, $D$ can be empty).

Just as in the previous proof, we claim that the $\O_K$-lattice 
$L = D\perp T\otimes\O_K$ 
represents all totally positive elements of norm less then $\NN(n)$: It suffices to check this for squarefree elements, and each of those is represented either by $T \otimes \O_K$ or by $D$.

It remains to show that $L$ does not represent $n$. For contradiction, assume that it does, i.e., 
$$n=\sum_{i=1}^k\alpha_ix_i^2+T(y_1,\dots,y_\ell), \text{ where } x_i,y_j\in\O_K.$$ 

By assumption (a), we have $\alpha_ix_i^2\in\{0,1,2,\dots,n\}$. By assumption (b) we then have 
$x_i=\epsilon_it_i, \text{ with } \epsilon_i\in \U_K,t_i\in\Z.$
If $t_i\neq 0$ for some $i$, then we have $\alpha_i\epsilon_i^2\in\Z$ with $\epsilon_i\in \U_K$, contradicting our choice of $\alpha_i$. Thus $t_i=0$ for all $i$ and $n=T(y_1,\dots,y_\ell), \text{ where } y_j\in\O_K.$

Considering \eqref{eq:2T} and assumption (a), we see that $L_h(y_1,\dots,y_\ell)^2=\left(\sum_{j=1}^\ell a_{hj} y_j\right)^2\in\Z$ for every $h$.
Using assumption (b), we get $\sum_{j=1}^\ell a_{hj} y_j=s_h\in\Z$ (similarly as in the proof of Proposition \ref{pr:15iscritical}).

Choose an integral basis $\omega_1=1,\omega_2,\dots,\omega_d$ for $\O_K$ and write $y_j=\sum_{f=1}^d y_{jf}\omega_f$ with $y_{jf}\in\Z$ so that we have 
$\sum_{j=1}^\ell a_{hj} \sum_{f=1}^d y_{jf}\omega_f=s_h$.

The elements $\omega_1=1,\omega_2,\dots,\omega_d$ are $\Q$-linearly independent, hence 
$\sum_{j=1}^\ell a_{hj}  y_{j1}=s_h$. But this means that $$T(y_{11},\dots,y_{\ell 1})=\frac12 \sum_{h=1}^g L_h(y_{11},\dots,y_{\ell 1})^2=\frac12 \sum_{h=1}^g s_h^2=n,$$ a contradiction as we assumed that $n$ was not represented by $T$ over $\Z$.
\end{proof}

In the previous two statements, we have studied when inclusions of the type $\CC^X_\Q \subset \CC^X_K  \cap \Q$ hold. Going in the other direction is significantly easier.

\begin{proposition}\label{prop:other inclusion} Let $K$ be a totally real number field.
 \begin{enumerate}
     \item[(a)] Assume that for all $\beta \in \O_K$ we have: If $\beta^2 \preceq 4 \cdot 203 \cdot 290$, then $\beta \in \Z$. Then $\CC_K \cap \Q \subset \CC_\Q$.
     \item[(b)] Assume that for all $\beta \in \O_K$ we have: If $\beta^2 \preceq 14 \cdot 15$, then $\beta \in \Z$. Then $\Ccl_K \cap \Q \subset \Ccl_\Q$.
 \end{enumerate}
\end{proposition}
\begin{proof}
We prove (a); the proof of (b) is analogous. Consider $n \in \CC_K \cap \Q$. Then there is an $\O_K$-lattice $L$ with truant $n$. Let $\CC_\Q = \{a_1, a_2, \ldots, a_{29}\}$ be numbered in the increasing order, so $a_{28}=203$ and $a_{29} = 290$. For every $a_i < n$, pick a vector $v_i \in L$ such that $Q(v_i) = a_i$. Consider now $B(v_i,v_j)$ for every distinct pair $v_i, v_j$. By the Cauchy--Schwarz inequality, we have $\bigl(2B(v_i,v_j)\bigr)^2 \leq 4 a_i a_j \leq 4a_{28} a_{29} = 4 \cdot 203 \cdot 290$, so by assumption, $B(v_i,v_j) \in \tfrac12\Z$. This means that the $\Z$-span of the $v_i$ is a well-defined quadratic $\Z$-lattice; denote it by $L'$. We also see that $L'$ has truant $n$: It clearly does not represent $n$. On the other hand, the truant of a $\Z$-lattice is an element of $\CC_\Q$, and all elements of $\CC_\Q$ smaller than $n$ are represented by $L'$ by construction. Thus the truant of $L'$ is $n$, which yields $n \in \CC_\Q$.
\end{proof}

For concreteness, observe that one can easily establish explicit sufficient conditions 
for 
the assumptions from Propositions \ref{pr:15iscritical}, \ref{pr:290iscritical}, and \ref{prop:other inclusion} to hold.

\begin{lemma} \label{le:47}
    Let $K$ be a totally real number field and let $m\in \Z^+$.
 \begin{enumerate}
     \item[(a)] Assume that there are no proper intermediate fields $\Q\subsetneq F\subsetneq K$ and that $\mathrm{disc}_K$ is sufficiently large (depending on $m$ and $[K:\Q]$). Then:
     
     If $\alpha \preceq m$ for $\alpha \in \OPlus_K$, then $\alpha \in \Z$.
     
     If $\beta^2 \preceq m$ for $\beta \in \O_K$, then $\beta \in \Z$.

     \item[(b)] Assume that every prime that divides $m$ is inert in $K$. If $\alpha,\omega\in\O_K$ satisfy $m=\alpha\omega^2$, then there exists $\epsilon \in \U_K$ such that $\epsilon\omega \in \Z$.
 \end{enumerate}
\end{lemma}

\begin{proof} Let $d=[K:\Q]$ be the degree of $K$ over $\Q$, and let us assume that $d>1$, as the statement is trivial when $K=\Q$.

(a) If $\alpha \preceq m$ for $\alpha \in \OPlus_K$, then $0<\Tr(\alpha)\leq \Tr(m)=md$. By the Stieltjes--Schur Theorem \cite[§3]{Schur}, the trace of $\alpha$ grows with the discriminant $\Delta_\alpha$ of the element $\alpha$, which is greater or equal to the discriminant of the number field $\Q(\alpha)$. If we assume for contradiction that $\alpha \not \in \Z$, then necessarily $\Q(\alpha)=K$ as there are no proper intermediate fields, and then we have
\begin{equation}\label{eq:schur}
md\geq \Tr(\alpha)\geq c_d \Delta_\alpha^{1/(d^2-d)}\geq c_d \mathrm{disc}_K^{1/(d^2-d)},
\end{equation}
where $c_d$ is an explicit constant from the Stieltjes--Schur Theorem. This is a contradiction when $\mathrm{disc}_K$ is sufficiently large.

Concerning $\beta$, one just uses another Stieltjes--Schur inequality (that says $\Tr(\beta^2)\geq c_d' \Delta_\beta^{2/(d^2-d)}$ for some $c_d'>0$, see \cite[Proposition 2]{Ka3}) in the preceding argument.

    (b) Let $m=\prod_i p_i^{k_i}$ be the prime factorization of $m$ in $\Z$. By the assumption, each $p_i$ is inert in $K$, i.e., $(p_i)$ are pairwise distinct prime ideals in $\O_K$. As $\omega^2\mid m$, we have $(\omega)=\prod_{i} (p_i)^{\ell_{i}}$ with $2\ell_{i}\leq k_i$. Thus there is a unit $\epsilon\in\U_K$ such that $\epsilon\omega=\prod_{i} p_i^{\ell_{i}}\in\Z$.
\end{proof}

One can also generalize the preceding arguments to sometimes show that  $\Cdiag_F \subset \Cdiag_K$ for an arbitrary extension of totally real number fields, except that one needs to add an assumption (c) that was automatically satisfied for $F=\Q$. The proof of the following result is otherwise the same as the proof of Proposition \ref{pr:15iscritical}, and so we omit it.

\begin{proposition} \label{pr:cheapgeneralization}
Let $K \supset F$ be totally real number fields. Let $a\U_F^2 \in \Cdiag_F$. 
Assume that:
 \begin{enumerate}
     \item[(a)] 
     If $\alpha \preceq a$ for $\alpha \in \OPlus_K$, then $\alpha \in \O_F$.
     \item[(b)] If $b \in \O_F^+, b\preceq a,$ can be written as $b=\alpha\omega^2$ with $\alpha,\omega\in\O_K$, then there exists $\epsilon \in \U_K$ such that $\epsilon\omega \in \O_F$.
     \item[(c)] If $\epsilon^2\in\U_F$ for $\epsilon\in\U_K$, then $\epsilon\in\U_F$.
 \end{enumerate}
Then $\Cdiag_F \cap \{b \in \O_F^+\mid  b\preceq a\} \subset \Cdiag_K$.    
\end{proposition}

Finally, note that the argument from the proof of Proposition \ref{prop:other inclusion} also yields the following statement.

\begin{proposition}\label{prop:other F}
    Let $K \supset F$ be totally real number fields. Fix representatives $a_1,\dots,a_t\in\O_F^+$ for $\CC_F$.
    
    Assume that for all $\beta \in \O_K$ and for all $i\neq j$ 
    we have: If $\beta^2\preceq 4a_ia_j$, then $\beta\in\O_F$. Then $\CC_K \cap F \subset \CC_F$.
\end{proposition}

As in Proposition \ref{prop:other inclusion}, we can also easily get a statement concerning $\Ccl_K \cap F \subset \Ccl_F$, but \textit{not} a statement for $\Cdiag_K \cap F \subset \Cdiag_F$, as a sublattice of a diagonal lattice is not necessarily diagonal.

We can again arrange for the assumption (a) of Proposition \ref{pr:cheapgeneralization} and the assumption of Proposition \ref{prop:other F} to be satisfied by using Stieltjes--Schur inequality as in the proof of Lemma \ref{le:47}. We just replace $md$ on the left-hand side of \eqref{eq:schur} by $\Tr(a)$ or $\max\bigl(\Tr(4a_ia_j)\bigr)$; we also need to assume that all subfields $E\subsetneq K$ satisfy $E\subset F$. The only problem is that this yields a less explicit bound on $\mathrm{disc}_K$ (for it depends on traces of elements in the criterion set).

\section{Variations of the criterion set definition} \label{se:weaker}

We conclude the paper by considering two alternative natural definitions of criterion sets, and in both cases we show that they do not affect their structure.

First, in Theorem \ref{th:free} we note that if one works with free quadratic lattices (that directly correspond to quadratic forms) instead of all lattices, the criterion sets do not change. Although this is almost a direct corollary of Theorem \ref{th:unique}, we consider it somewhat surprising: Switching from forms to lattices makes no difference, whereas switching from classical to non-classical lattices changes the size of the set considerably (compare the 15- and 290-Theorems).

\begin{theorem} \label{th:free}
Let $\CC \subset S$ and $X\in\{\cl,\nc\}$. Then $\CC$ is an $(S,X)$-criterion set if and only if the following equivalence holds for every \emph{free} quadratic $X$-lattice $L$: $L$ is $S$-universal if and only if $L$ is $\CC$-universal. 
\end{theorem}

In other words, the $(S,X)$-criterion set for free lattices is the same as the $(S,X)$-criterion set for all lattices. In particular, if $S = \OO$, we get $\Cfr_K=\CC_K$ and $\Cclfr_K= \Ccl_K$.  
As  diagonal lattices are free by definition, the case $X=\diag$ is not included in the previous theorem.

\begin{proof}
Trivially, an $(S,X)$-criterion set for lattices is also an $(S,X)$-criterion set for free quadratic lattices. Let now $\CC \subset S$ be some $(S,X)$-criterion set for free quadratic lattices. By Theorem \ref{th:unique}, it is enough to show that $\CC$ contains all $(S,X)$-critical elements.

Consider an $(S,X)$-critical element $\alpha$ and take the $\bigl(S \setminus\{\alpha\}\bigr)$-universal lattice $L$ not representing $\alpha$. Take arbitrary $\beta \in \OPlus_K$ with $\NN(\beta) > \NN(\alpha)$, e.g., $\beta=\alpha+1$. Then consider the lattice $L' = L \perp \mathfrak{a}\cdot\qf{\beta}$ where $\mathfrak{a}$ is an integral ideal such that $L'$ is a free lattice (i.e., $\mathfrak{a}$ lies in the same class as $\mathfrak{b}^{-1}$ where $\mathfrak{b}$ is the Steinitz class of $L$, i.e., $L=\O_K v_1 + \cdots + \O_K v_{n-1} + \mathfrak{b}v_n$). By construction, $L'$ is free (see \cite[Theorem 1.39]{Nar}) and $\bigl(S \setminus\{\alpha\}\bigr)$-universal; and since $\mathfrak{a}\cdot\qf{\beta}$ represents only elements with $\NN(\,\cdot\,)>\NN(\alpha)$, we again get $\alpha \not\to L'$.

Thus every $(S,X)$-criterion set $\CC$ for free quadratic lattices must contain $\alpha$.
\end{proof}

Next, we will discuss the following variation of the definition of a criterion set: For a set $\CC \subset \OO$, not necessarily $\CC \subset S$, and for $X\in\{\mathrm{cl},\mathrm{nc},\mathrm{diag}\}$ we say that it is a \emph{generalized $(S,X)$-criterion set} if for every quadratic $X$-lattice $L$ we have: \enquote{$L$ is $S$-universal if and only if it is $\CC$-universal}. Namely, the only change is that we no longer require $\CC \subset S$.

Note that this variation makes no difference for the arguably main case $S=\OO$. Also observe that only after this variation, it started to be important to state the definition as an equivalence. As long as we assumed $\CC \subset S$, the implication \enquote{$S$-universal $\Rightarrow$ $\CC$-universal} was trivial. However, here it does play a role, as we now illustrate.

\begin{example}
Let $K=\Q$ and $S = \{(2k)^2 : k \in \N\}$. Clearly each lattice representing $1$ is $S$-universal, but it is not difficult to see that there exists a quadratic lattice which represents all of $S$ but not $1$, for example $\qf{4,4,4,4}$. Thus $\{1\}$ is not a generalized $(S,X)$-criterion set, for each $X\in\{\mathrm{cl},\mathrm{nc},\mathrm{diag}\}$.
\end{example}

In the previous example, if we want to know whether a given lattice is $S$-universal, it is natural to start by checking whether it represents $1$, since this simple test can possibly give us the full answer. Hence $\{1\}$ provides some information about the $S$-universality of lattices. However, for a set to be called (generalized) $(S,X)$-criterion set, this set must capture full information about $S$-universal lattices, and this is not the case of $\{1\}$.

Clearly, there is a difference between the notions of an \emph{$(S,X)$-criterion set} and a \emph{generalized $S$-criterion set}: Each set containing an $(S,X)$-criterion set but not contained in $S$ belongs to the latter, but not the former category. Now we show that this trivial difference is the only one.

\begin{theorem} \label{th:generalized}
Let $S,\CC \subset \OO$ and $X\in\{\cl,\nc\}$. If $\CC$ is a generalized $(S,X)$-criterion set, then it contains $\CC_S^X$, the set of all $(S,X)$-critical elements. In particular, the minimal $(S,X)$-criterion set $\CC_S^X$ is also the unique minimal generalized $(S,X)$-criterion set.
\end{theorem}

Before proving this statement, we observe that it yields full characterization of generalized $S$-criterion sets. Namely, there exists (quite clearly) the following maximal generalized $S$-criterion set:

Denote by $\mathcal H_S^X \subset \OO$ the set consisting of all those (square classes of) elements which are represented by every $S$-universal $X$-lattice. I.e., if $D_L$ is the set of all elements represented by a given $X$-lattice $L$ and $D_L / \U_K^2 = \{\alpha \U^2_K \mid \alpha \in D_L\}$ is the corresponding set modulo squares of units, we can write
\[
\mathcal H_S^X = \bigcap_{\text{$L$ is $S$-universal}} D_L / \U_K^2.
\]

Then by definition, every generalized $(S,X)$-criterion set is contained in $\mathcal H_S^X$, so we have (assuming Theorem \ref{th:generalized}) the following characterization: \enquote{$\CC$ is a generalized $(S,X)$-criterion set if and only if $\CC_S^X\subset \CC \subset \mathcal H_S^X$.} For now, we only know the latter inclusion.

In general, determining the full set $\mathcal H_S^X$ seems a very hard problem for almost any given $S$. For example, DeBenedetto \cite{De} showed that every prime-universal quadratic form over $\Z$ represents $205$; thus, in our notation, for $S = \mathbb{P} = \{\text{primes}\}$, we have $205 \in \mathcal H_{\mathbb{P}}^X$. Theorem \ref{th:generalized}, which we are about to prove, shows for example that whenever $\CC$ is a generalized $(\mathbb{P},X)$-criterion set, then so is $\CC \setminus \{205\}$. There is only one minimal generalized $(\mathbb{P},X)$-criterion set and it is a subset of $\mathbb P$.

\begin{proof}[Proof of Theorem $\ref{th:generalized}$]
Given a set $S \subset \OO$, let us consider the set $\mathcal H_S^X$ defined above and  denote it just as $\mathcal H$ throughout the proof. By definition, a lattice is $S$-universal if and only if it is $\mathcal H$-universal. Thus, the following are equivalent (for a set $\CC \subset \OO$):
 \begin{enumerate}[label=\alph*)]
     \item[(a)] $\CC$ is a generalized $(S,X)$-criterion set.
     \item[(b)] $\CC$ is a generalized $(\mathcal H,X)$-criterion set.
 \end{enumerate}
However, every generalized $(S,X)$-criterion set is contained in $\mathcal H$, so there is a further equivalent condition:
 \begin{enumerate}
     \item[(c)] $\CC$ is an $(\mathcal H,X)$-criterion set.
 \end{enumerate}
Using this condition and applying Theorem \ref{th:unique} for $\mathcal H$ yields the existence of a set $\mathcal C_{\mathcal H}^X$ such that all the above conditions are equivalent to
 \begin{enumerate}
     \item[(d)] $\mathcal C_{\mathcal H}^X \subset \CC \subset \mathcal H$.
 \end{enumerate}
This set $\mathcal C_{\mathcal H}^X$ can be described as the set of all $(\mathcal H,X)$-critical elements, but that is not important. We will show that $\CC_S^X = \mathcal C_{\mathcal H}^X$:

Since $S$ itself is a generalized $(S,X)$-criterion set, we have $\mathcal C_{\mathcal H}^X \subset S$. This shows that not only is $\mathcal C_{\mathcal H}^X$ a generalized $S$-criterion set thanks to satisfying (d), but actually an ``honest'' $(S,X)$-criterion set. Since $\CC_S^X$ is the unique smallest $(S,X)$-criterion set, this implies $\CC_S^X \subset \mathcal C_{\mathcal H}^X$. On the other hand, since $\CC_S^X$ is a generalized $(S,X)$-criterion set, (d) yields $\mathcal C_{\mathcal H}^X \subset \CC_S^X$. Thus, indeed, $\CC_S^X = \mathcal C_{\mathcal H}^X$. So by (d), every generalized $(S,X)$-criterion set contains $\CC_S^X$.
\end{proof}

As the maximal criterion sets $\mathcal H_S^X$ correspond bijectively to the \textit{finite} minimal sets $\mathcal C_S^X$, we see that there are only countably many different maximal criterion sets (even though there are uncountably many sets $S\subset\OO$). Given $S$, explicitly determining either the minimal or maximal criterion set is a difficult task for which no algorithm is known (although see Remark \ref{rem:34}).

\section{Open problems}\label{se:6}

Let us conclude the paper by listing several open problems, which may stimulate rich further research.

\subsection{Growth of criterion sets}\label{subsec:growth}

\textit{How does the size $\#\CC_K$ grow with the field $K$ (say, in terms of its discriminant)?}
Corollary \ref{pr:quadratic} implies that 
$\#\CC_K \geq m_K$, and it is known that $m_K$ is often large (see the discussion in the introduction), and so also $\#\CC_K$ is often large.

However, one may be tempted to ask about even stronger results. 

\begin{question}\label{que1}
    Fix positive integers $d$ and $c$. Are there only finitely many totally real fields $K$ of degree $[K:\Q]=d$ such that $\#\CC_K\leq c$?
\end{question}

Theorem \ref{th:indecomposable} shows that squarefree indecomposables lie in criterion sets. The results on the geometry of continued fractions and sails \cite{KM, Man} suggest that, when $d=[K:\Q]\geq 3$, the number of indecomposables may grow with the discriminant of $K$. One then would expect also the number of squarefree indecomposables to grow, leading us to expect that Question \ref{que1} has positive answer when $d\geq 3$.

However, the situation is different in the case $d=2$ of real quadratic fields, as there are infinite families of quadratic fields with very few indecomposables. For example, if $n$ is a positive integer such that $n^2-1$ is squarefree, then all indecomposables in $\Q(\sqrt{n^2-1})$ are units, so Theorem \ref{th:indecomposable} contributes only two elements of $\OO$ to $\Cdiag_K$, namely $1$ and the fundamental unit $\epsilon$. In a work in progress, \cite{KM+} show that these fields indeed have bounded sizes of $\CC_K^X$ -- but the case of other real quadratic fields is much less clear.

\subsection{Explicit criterion sets}\label{subsec:expl}

Criterion sets $\CC_K^X$ are known only in very few cases: Besides the sets of 15- and 290-Theorems over $\Z$, the only other result has been Lee's determination of $\Ccl_{\Q(\sqrt 5)}$ \cite{Le}. The follow-up paper \cite{KR} will obtain several (mostly conjectural) results for $\Q(\sqrt 2), \Q(\sqrt 3), \Q(\sqrt 5)$, 
but using the methods of Proposition \ref{pr:critical} and Remark \ref{rem:34}, it should be possible to \textit{computationally generate candidates for criterion sets for many more (particularly real quadratic) fields}. These results would provide evidence for Question \ref{que1}.

More generally, one should \textit{search for sets $\CC_S^X$ for various natural choices of $S$}.

\subsection{Provable criterion sets}

Once equipped with conjectural data from \ref{subsec:expl}, it would be tempting to prove some of these results. At least when dealing with a single field $K$ with $m_K\geq 5$, one may attempt to use the asymptotic local-global principle  \cite{HKK} (and its explicit version due to Hsia--Icaza \cite{HI}) for this purpose, although the bounds will probably be too large to be computationally feasible.

One may also try to answer the following question theoretically, at least for some family of fields (e.g., those having odd discriminant and sufficiently many square classes of totally positive units \cite{KM+}).

\begin{question}
Can we find an explicit bound $B_K$ such that all elements of $\CC_K$ must have $N(\alpha)<B_K$?
\end{question}

\subsection{Very small criterion sets}

Concerning Question \ref{que1}, for very small values of $c$, we suspect that there are only finitely many such fields $K$ overall, without restricting the degree $d$. \textit{Can one find all fields with $\#\CC_K\leq 6$, say?}

\subsection{Relation between different criterion sets} 
In Theorem \ref{th:unique} we proved that $\Cdiag_K \subset \Ccl_K \subset \CC_K$; apriori, we do not know anything more about the relations between these sets. It is thus somewhat surprising that $\Cdiag_\Q = \Ccl_\Q$ (while $\CC_\Q$ is much larger) \cite{Bh,BH}. The computations \cite{KR} show that $\#\Cdiag_{\Q(\sqrt 5)} = \#\Ccl_{\Q(\sqrt 5)}-1$, and so these sets can be different. However, it seems at least plausible that there is some general relationship between them, perhaps related to the fact that sublattices of diagonal lattices are classical. 
\textit{Is there some precise sense in which $\Cdiag_K$ and $\Ccl_K$ have comparable size?}

\subsection{Lifting of criterion sets}

Let $F\subsetneq K$ be totally real fields. 
The results of Section \ref{se:indec} show various partial relations between criterion sets for $F$ and $K$. But they leave open the following question in the spirit of the lifting problem \cite{KY1}: \textit{Can $\O_F^+$ be a criterion set for $K$?}
It is quite likely that the answer is always \textit{No!}; we can at least easily show the following.

\begin{theorem}
    Let $F=\Q$ or $\Q(\sqrt 5)$ and let $K\supsetneq F$ be a totally real field. Then $\O_F^+$ is not an $X$-criterion set for $K$ (for $X\in\{\nc,\cl,\diag\}$).
\end{theorem}

\begin{proof}
    For the sake of contradiction, assume that $\O_F^+$ is an $X$-criterion set for $K$. As the sum of four squares $\qf{1,1,1,1}$ is universal over $\O_F^+$ for $F=\Q$ and $\Q(\sqrt 5)$ \cite{M}, it must be also universal over $K$ (as we are assuming that $\O_F^+$ is an $X$-criterion set).
    But this happens only when $K=\Q$ and $\Q(\sqrt 5)$ by Siegel's theorem \cite{Si}. As $K\supsetneq F$, this leaves only the case $F=\Q$, $K=\Q(\sqrt 5)$. However, $\qf{1,1,3,3}$ represents all elements of $\Z^+$, but one easily checks that it does not represent $(7+\sqrt 5)/2$ over $\O_K$.   
\end{proof}

\subsection{Critical elements}

Definition \ref{de:truant} and Proposition \ref{pr:critical}
characterized critical elements in terms of lattices that represent all elements of smaller norm; this later turned out to be the key in proving uniqueness of minimal criterion sets. However, for some purposes, it may be more natural to order elements by their trace or house (the \textit{house} of an algebraic number is defined as the maximum of the absolute values of its conjugates). \textit{Does a variant of Proposition $\ref{pr:critical}$ hold if we define truants in terms of trace or house?} The argument given in the proof of Proposition \ref{pr:critical} breaks, but that does not mean that the statement is not true.

\subsection{Maximal criterion sets}\label{subsec:max}

One can use the method of escalations to generate a list of small elements that must lie in the ``maximal'' sets $\mathcal H_S^X$. \textit{Is there a more theoretical way of understanding their elements?} Again, first obtaining significant computational data as in \ref{subsec:expl} may be very helpful.

\section*{Declarations}

\textbf{Data sharing:} Not applicable to this article as no datasets were generated or analysed during the current study.

\noindent \textbf{Competing interests:} The authors have no competing interests to declare that are relevant to the content of this article.

\noindent \textbf{Funding:} 
V. K. was supported by Czech Science Foundation [grant numbers 21-00420M, 26-20514S].
G. R. was supported by Charles University Primus Programme [grant number 25/SCI/008].

\section*{Acknowledgments} We thank the members of \href{https://www1.karlin.mff.cuni.cz/~kala/web/ufoclan}{UFOCLAN} group for numerous interesting discussions about the topics of the paper. We are very grateful to the anonymous referee for several helpful comments and suggestions.


\begin{thebibliography}{WWW}
{\small

\bibitem[Be1]{Be1} C. N. Beli, \textit{A new approach to classification of integral quadratic forms over dyadic local fields}, Trans. Amer. Math. Soc. \textbf{362} (2010), 1599--1617.

\bibitem[Be2]{Be2} C. N. Beli, \textit{Universal integral quadratic forms over dyadic local fields}, \href{https://arxiv.org/abs/2008.10113}{arxiv:2008.10113}.

\bibitem[Bh]{Bh} M. Bhargava, \emph{On the Conway-Schneeberger fifteen theorem}, Contemp. Math. \textbf{272} (1999), 27--37.

\bibitem[BH]{BH} M. Bhargava, J. Hanke, \emph{Universal quadratic forms and the 290-theorem}, 2011, preprint, \url{https://math.stanford.edu/~vakil/files/290-Theorem-preprint.pdf}


	\bibitem[BK]{BK1} V. Blomer,  V. Kala, {\em Number fields without universal $n$-ary quadratic forms}, Math. Proc. Cambridge Philos. Soc. \textbf{159} (2015), 239--252.



		\bibitem[BK2]{BK2} V. Blomer, V. Kala, \emph{On the rank of universal quadratic forms over real quadratic fields}, Doc. Math. \textbf{23} (2018), 15--34.
		

\bibitem[BC]{BC} M. Bordignon, G. Cherubini, \textit{Coprime-universal quadratic forms}, \href{https://arxiv.org/abs/2406.01533}{arxiv:2406.01533}.


		\bibitem[CL+]{CL+} M. \v Cech, D. Lachman, J. Svoboda, M. Tinkov\' a, K. Zemkov\' a, \emph{Universal quadratic forms and indecomposables over biquadratic fields}, Math. Nachr. \textbf{292} (2019), 540--555.
		

\bibitem[CKR]{CKR} W. K. Chan, M.-H. Kim, S. Raghavan, \emph{Ternary universal integral quadratic forms}, Japan. J. Math. \textbf{22} (1996), 263--273.


	\bibitem[CO]{CO} W. K. Chan, B.-K. Oh, {\em Can we recover an integral quadratic form by representing all its subforms?}, Adv. Math. \textbf{433} (2023), Paper No. 109317.


\bibitem[De]{De} J. DeBenedetto, \textit{Quadratic forms representing all primes}, Involve \textbf{7} (2014), 619--626. 

\bibitem[DR]{DR} 
J. DeBenedetto, J. Rouse, \textit{Quadratic forms representing all integers coprime to 3}, Ramanujan
J. \textbf{46} (2018),  431--446.

		\bibitem[DS]{DS} A. Dress, R. Scharlau, \emph{Indecomposable totally positive numbers in real quadratic orders}, J. Number Theory \textbf{14} (1982), 292--306.



\bibitem[Ea]{Ea} A. G. Earnest,\emph{Universal and regular positive quadratic lattices over totally real number fields}, Contemp. Math. \textbf{249}  (1999), 17--27.


	\bibitem[EK]{EK} A. G. Earnest, A. Khosravani, {\em Universal positive quaternary quadratic lattices over totally real number fields}, Mathematika \textbf{44} (1997), 342--347.

 \bibitem[EKK]{EKK} N. D. Elkies, D. M. Kane, S. D. Kominers, \textit{Minimal $S$-universality criteria may vary in size}, J. Théor. Nombres Bordeaux \textbf{25} (2013), 557--563.


\bibitem[HH]{HH} Z. He, Y. Hu, 
\textit{On $n$-universal quadratic forms over dyadic local fields},
Sci. China Math. \textbf{67} (2024), 1481--1506.

\bibitem[HHX]{HHX} Z. He, Y. Hu, F. Xu, \emph{On indefinite {$k$}-universal integral quadratic forms over number fields}, Math. Z. \textbf{304} (2023), Paper No. 20. 


\bibitem[HK]{HK} T. Hejda, V. Kala, \textit{Additive structure of totally positive quadratic integers}, Manuscripta Math. \textbf{163}  (2020), 263--278.

\bibitem[HI]{HI} J. S. Hsia, M. I. Icaza, \textit{Effective version of Tartakowsky's theorem},
Acta Arith. \textbf{89} (1999), 235--253.

\bibitem[HKK]{HKK} J. S. Hsia, Y. Kitaoka, M. Kneser, \textit{Representations of positive definite quadratic forms}, J. Reine Angew. Math. \textbf{301} (1978), 132--141.

		
		\bibitem[JK]{JK} S. W. Jang, B. M. Kim, \emph{A refinement of the Dress-Scharlau theorem}, J. Number Theory \textbf{158} (2016), 234--243.

\bibitem[JuK]{JuK} J. Ju, D. Kim, \textit{
The pentagonal theorem of sixty-three and generalizations of Cauchy’s lemma}, Forum Math. \textbf{35} (2023),  1685--1706.

	\bibitem[Ka1]{Ka1} V. Kala, {\em Universal quadratic forms and elements of small norm in real quadratic fields}, Bull. Aust. Math. Soc. \textbf{94} (2016), 7--14.

 
		
		
  \bibitem[Ka2]{Ka1.5} V. Kala, \emph{Norms of indecomposable integers in real quadratic fields}, J. Number Theory \textbf{166} (2016), 193--207.
	
	\bibitem[Ka3]{Ka2} V. Kala, {\em Universal quadratic forms and indecomposables in number fields: A survey}, Commun. Math. \textbf{31} (2023), 81--114.

    \bibitem[Ka4]{Ka3} V. Kala, \textit{Number fields without universal quadratic forms of small rank exist in most degrees}, Math. Proc. Cambridge Philos. Soc. \textbf{174} (2023), 225--231.


\bibitem[KKK]{KKK} V. Kala, K. Kramer, J. Kr\' asensk\' y, \textit{Kitaoka's Conjecture and sums of squares}, \href{https://arxiv.org/abs/2510.19545}{arxiv:2510.19545}.

\bibitem[KK+]{KK+}  V. Kala, J. Kr\' asensk\' y, D. Park, P. Yatsyna, B. \. Zmija, \textit{Kitaoka's conjecture for quadratic fields}, \href{https://arxiv.org/abs/2501.19371}{arxiv:2501.19371}.

\bibitem[KM]{KM} V. Kala, S. H. Man, \textit{Sails for universal quadratic forms}, Selecta Math. (N.S.) \textbf{31} (2025), Paper No. 26.

\bibitem[KM+]{KM+} V. Kala, S. H. Man, R. Visser, P. Yatsyna, \textit{Explicit criterion sets for universal quadratic forms}, in preparation.

\bibitem[KP]{KP} V. Kala, O. Prakash, \textit{There is no 290-Theorem for higher degree forms}, Math. Nachr. \textbf{297} (2024), 4322--4332.
    
	\bibitem[KT]{KT} V. Kala,  M. Tinkov\'a, {\em Universal quadratic forms, small norms and traces in families of number fields}, Int. Math. Res. Not. IMRN (2023), 7541--7577.

	\bibitem[KY1]{KY1} V. Kala,  P. Yatsyna, {\em Lifting problem for universal quadratic forms}, Adv. Math. \textbf{377} (2021), Paper No.  107497.

\bibitem[KY2]{KY} V. Kala, P. Yatsyna, {\em On {K}itaoka's conjecture and lifting problem for universal quadratic forms}, Bull. Lond. Math. Soc., \textbf{55} (2023), 854--864. 



	\bibitem[KYZ]{KYZ} V. Kala, P. Yatsyna,  B. \.Zmija, {\em Real quadratic fields with a universal form of given rank have density zero}, Amer. J. Math. (to appear), \href{https://arxiv.org/abs/2302.12080}{arxiv:2302.12080}.

\bibitem[KaL]{KaL} B. Kane, J. Liu, \textit{Universal sums of $m$-gonal numbers}, Int. Math. Res. Not. IMRN (2020),  6999--7036.
 
	
	\bibitem[Ki]{Ki2} B. M. Kim, {\em Universal octonary diagonal forms over some real quadratic fields}, Comment. Math. Helv. \textbf{75} (2000), 410--414. 


\bibitem[KKO1]{KKO} B. M. Kim, M.-H. Kim, B.-K. Oh, \textit{2-universal positive definite integral
quinary quadratic forms}, Contemp. Math. \textbf{249} (1999), 51--62.


\bibitem[KKO2]{KKO2} B. M. Kim, M.-H. Kim, B.-K. Oh, \textit{A finiteness theorem for representability of quadratic forms by forms}, J. Reine Angew. Math. \textbf{581} (2005), 23--30.

\bibitem[KKP1]{KKP0} B. M. Kim, J. Y. Kim, P.-S. Park,
		\emph{The fifteen theorem for universal Hermitian lattices over imaginary quadratic fields},
		Math. Comp. \textbf{79} (2010), 1123--1144.
	
\bibitem[KKP2]{KKP} B. M. Kim, M.-H. Kim,  D. Park, {\em Real quadratic fields admitting universal lattice of rank $7$}, J. Number Theory \textbf{233} (2022), 456--466.
  
    \bibitem[KL]{KL} D. Kim, S. H. Lee, \textit{Lifting problem for universal quadratic forms over totally real cubic number fields},  Bull. Lond. Math. Soc. \textbf{56} (2024), 1192--1206.

\bibitem[KLO]{KLO} K. Kim, J. Lee, B.-K. Oh, \textit{Minimal universality criterion sets on the representations of binary quadratic forms}, 
J. Number Theory \textbf{238} (2022), 37--59.

\bibitem[Kim]{Ki} M.-H. Kim, \emph{Recent developments on universal forms}, Contemp. Math. \textbf{344} (2004), 215--228.


\bibitem[Ko]{Ko} C. Ko, \emph{On the representation of a quadratic form as a sum of squares of linear forms}, Q. J. Math. \textbf{1} (1937), 81--98.
		

\bibitem[Kom1]{Kom} S. D. Kominers, \textit{Uniqueness of the 2-universality criterion}, 
Note Mat. \textbf{28} (2008),  203--206.

\bibitem[Kom2]{Kom2} S. D. Kominers, 
\textit{Oh's 8-universality criterion is unique}, 
Kyungpook Math. J. \textbf{61} (2021), 455--459.


\bibitem[KR]{KR} J. Kr{\'a}sensk{\'y}, G. Romeo, \textit{Escalations and criteria over real quadratic fields}, in preparation.
	
	\bibitem[KTZ]{KTZ} J. Kr{\'a}sensk{\'y}, M. Tinkov{\'a},  K. Zemkov{\'a}, {\em There are no universal ternary quadratic forms over biquadratic fields}, Proc. Edinb. Math. Soc. \textbf{63} (2020), 861--912.

\bibitem[Lee]{Le} Y. M. Lee, \textit{Universal forms over $\Q(\sqrt5)$}, Ramanujan J. \textbf{16} (2008), 97--104.


\bibitem[Ma]{M} H. Maass, \emph{\"{U}ber die {D}arstellung total positiver {Z}ahlen des {K}\"{o}rpers {$R(5)$} als {S}umme von drei {Q}uadraten}, Abh. Math. Sem. Hamburg \textbf{14} (1941), 185--191.


    \bibitem[Man]{Man} S. H. Man, \textit{Minimal rank of universal lattices and number of indecomposable elements in real multiquadratic fields}, Adv. Math. \textbf{447} (2024), Paper No. 109694.

    \bibitem[Nar]{Nar} W. Narkiewicz, \emph{Elementary and analytic theory of algebraic numbers}, 3rd Edition, Springer-Verlag, Berlin, 2004.


\bibitem[Oh]{Oh} B.-K. Oh, \textit{Universal Z-lattices of minimal rank}, Proc. Amer. Math. Soc. \textbf{128} (2000), 683--689.

\bibitem [OM]{OM} O. T. O'Meara, {\em Introduction to quadratic forms}, Springer-Verlag, New York  (1963).



\bibitem[Ro]{Ro} J. Rouse, \emph{Quadratic forms representing all odd positive integers},
Amer. J. Math. \textbf{136} (2014), 1693--1745.


\bibitem[Sch]{Schur} I. Schur, \textit{\" Uber die Verteilung der Wurzeln bei gewissen algebraischen Gleichungen mit ganzzahligen Koeffizienten}, Math. Z. \textbf{1} (1918), 377--402.

\bibitem[Si1]{Si1} C. L. Siegel, \textit{The trace of totally positive and real algebraic integers}, Ann. of Math. \textbf{46} (1945), 302--312.
 
\bibitem[Si2]{Si} C. L. Siegel, {\em Sums of $m$-th powers of algebraic integers}, Ann. of Math. \textbf{46} (1945), 313--339.

\bibitem[St]{St} P. Stevenhagen, \textit{Radei reciprocity, governing fields and negative Pell}, Math. Proc. Cambridge Philos. Soc. \textbf{172} (2022), 627--654.


	\bibitem[Ti]{Ti} M. Tinkov\'a, \emph{Trace and norm of indecomposable integers in cubic orders}, Ramanujan J. \textbf{61} (2023), 1121--1144.
		
		
	\bibitem[TV]{TV} M. Tinkov\' a, P. Voutier, \emph{Indecomposable integers in real quadratic fields}, J. Number Theory \textbf{212} (2020), 458--482.

    \bibitem[XZ]{XZ} F. Xu, Y. Zhang, \textit{On indefinite and potentially universal quadratic forms over number fields}, Trans. Amer. Math. Soc. \textbf{375} (2022), 2459--2480.

	\bibitem[Ya]{Ya} P. Yatsyna, {\em A lower bound for the rank of a universal quadratic form with integer coefficients in a totally real field}, Comment. Math. Helv. \textbf{94} (2019), 221--239.

}
\end{thebibliography}
\end{document}